\documentclass [11pt]{amsart}  
\usepackage{a4}                 
\usepackage{paralist}           
\usepackage{graphicx}
\usepackage{amssymb}                               
\usepackage{float}
\usepackage{amsmath}
\usepackage{psfrag}
\usepackage{MnSymbol}
\usepackage[all]{xy}
\usepackage{rotating}

\usepackage{amsmath,amsthm}

\theoremstyle{plain}
\newtheorem{thm}{Theorem}[section]

\newtheorem{lem}[thm]{Lemma}
\newtheorem{prop}[thm]{Proposition}
\newtheorem{cor}[thm]{Corollary}

\theoremstyle{definition}
\newtheorem{defn}[thm]{Definition}
\newtheorem{exmp}[thm]{Example}
\newtheorem{ques}{Question}

\theoremstyle{remark}

\newtheoremstyle{TheoremNum}
        {\topsep}{\topsep}              
        {\itshape}                      
        {}                              
        {\bfseries}                     
        {.}                             
        { }                             
        {\thmname{#1}\thmnote{ \bfseries #3}}
    \theoremstyle{TheoremNum}

\DeclareMathOperator{\ab}{ab}

\DeclareMathOperator{\core}{core}
\DeclareMathOperator{\comp}{\textbf{K}}

\title{Transversals as generating sets in finitely generated groups}
\author{Jack Button, Maurice Chiodo, Mariano Zeron-Medina Laris}
\date{\today}

\begin{document}

\let\thefootnote\relax\footnotetext{2010 \textit{AMS Classification:} 20E99, 20F05}
\let\thefootnote\relax
\let\thefootnote\relax\footnotetext{The second author was partially funded by the Italian FIRB ``Futuro in Ricerca'' project RBFR10DGUA\_002 and the Swiss National Science Foundation grant FN PP00P2-144681/1.}
\let\thefootnote\relax\footnotetext{The third author did part of this work while writing his thesis at the University of Cambridge, supported by the Mexican National Council for Science and Technology.}

\begin{abstract}
We explore transversals of finite index subgroups of finitely generated groups. We show that when $H$ is a subgroup of a rank $n$ group $G$ and $H$ has index at least $n$ in $G$ then we can construct a left transversal for $H$ which contains a generating set of size $n$ for $G$, and that the construction is algorithmic when $G$ is finitely presented. We also show that, in the case where $G$ has rank $n \leq3$, there is a simultaneous left-right transversal for $H$ which contains a generating set of size $n$ for $G$. We finish by showing that if $H$ is a subgroup of a rank $n$ group $G$ with index less than $3 \cdot 2^{n-1}$, and $H$ contains no primitive elements of $G$, then $H$ is normal in $G$ and $G/H \cong C_{2}^{n}$.
\end{abstract}

\maketitle

\section{Introduction}

Let $H$ be a subgroup of $G$ (written $H<G$). A \textit{left transversal} for $H$ in $G$ is a choice of exactly one representative from each left coset of $H$. A \textit{right transversal} for $H$ is $G$ in defined in an analogous fashion. A \emph{left-right} transversal for $H$ is a set $S$ which is simultaneously a left transversal, and a right transversal, for $H$ in $G$. The existence of
a left or right transversal is clear (assuming the Axiom of
Choice) whereas it is not immediate that a left-right transversal always exists.
We gave a short proof of this in \cite{BCZ} for the case where $H$ is of finite index, as well as a brief historical discussion of this result.

Transversals are natural objects of study, especially when $H$ has finite
index in $G$. Moreover, finding generating sets for a group $G$ is a well
known problem in the case when $G$ is finitely generated. Therefore we can
ask: given a finitely generated group $G$ and a finite index subgroup $H$
of $G$, is there a (say) left
transversal for $H$ in $G$ which also generates $G$? 
In fact if $T$ is a left transversal for $H$ in $G$, then it is clear that
$T^{-1}$ (the set of inverses of all elements in $T$) is a right transversal 
of $H$ in $G$. Moreover $\langle S \rangle=\langle S^{-1} \rangle$ for any 
subset $S$ of a group $G$, so we need only consider left transversals
throughout. Jain asked Cameron this question under the added assumption that $G$ is a finite group and $H$ is corefree in $G$, meaning
that $\mbox{core}_G(H)=\{e\}$ ($\core_{G}(H)$ is the intersection of all conjugates of $H$ in $G$: $\core_{G}(H):= \bigcap_{g \in G} g^{-1}Hg$). Cameron showed in \cite{Camtrans} that in this
case a generating left transversal always exists 
(see also \cite[Problem 100]{Cameron}). The proof is short but relies on a
result \cite{Whi} of Whiston on \emph{minimal} generating sets (ones where
no proper subset generates) of the symmetric
group which uses the 
classification of finite simple groups (CFSG).

However there is an obvious necessary condition for a subgroup $H$ of $G$
to possess a generating left transversal,
which is that the
index $[G:H]$ must be at least $d(G)$ (the \emph{rank} of $G$), defined to be the minimal
number of generators for the finitely generated group $G$. In Theorem \ref{main} we show that
this condition is also sufficient, for $G$ any
finitely generated group and $H$ any subgroup of finite index. We can then
try to strengthen this result by examining whether $[G:H]\geq d(G)$ implies
that there exists a left-right transversal for $H$ that generates $G$. We
have not managed to establish this in general but we have shown in Theorem
\ref{rank3} that it is true if $d(G)\leq 3$.

Our main method of proof in Section \ref{shiftbox section} is a new technique which we call \emph{shifting boxes}. It involves using the transitive action of a group $G$ on the set of left (or right) cosets of a subgroup $H<G$ to systematically apply Nielsen transformations to a generating set of $G$, such that the resulting generators lie inside (or outside) particular desired cosets of $H$. We have found this technique to be very intuitive for developing proofs. Though similar to Schreier graphs, the technique of shifting boxes is much more useful in finding solutions to the problems we have considered. Many of our results can be reduced to the case of subgroups of free groups (Proposition \ref{reduce to free}). However, even in this restricted case, the use of Stallings graphs (Schreier graphs for free groups) leads to proofs which are longer and more complicated than those by shifting boxes.

An element of a rank $n$ group $ G$ is \emph{primitive} if it lies in some generating set of size $n$ for $G$. The location of primitive elements relative to cosets of subgroups is already an area of interest. Parzanchevski and Puder \cite{Puder} show that if $w\in F_{n}$ is a non-primitive element then there is a finite index subgroup $H<F_{n}$ such that the coset $wH$ does not contain any primitive elements. Taking $w=e$ gives a finite index subgroup containing no primitive elements.

By applying the technique of shifting boxes developed in Section \ref{shiftbox section}, we are able to show in Theorem \ref{classify} that if $G$ is a rank $n$ group, then the only subgroup of $G$ with index less than $3 \cdot 2^{n-1}$ that can contain \emph{no} primitive elements is $[G,G]G^{2}$, and even then this only occurs when $G/([G,G]G^{2}) \cong C_{2}^{n}$. This gives an exponential lower bound on the index of subgroups which contain no primitive elements. Beyond the special case $[G,G]G^{2}$, the lower bound of $3 \cdot 2^{n-1}$ can be sharp: the free group $F_{2}$ has a subgroup of index $3 \cdot 2^{2-1}$ containing no primitive elements. Moreover, $F_{n}$ always has a subgroup of index $2^{n+1}$ containing no primitive elements (Example \ref{2x2n eg}).

We first announced many of the results of this paper in \cite{BCZ2}.
\\ 
\\ \textbf{Acknowledgements:} We wish to thank Rishi Vyas and Andrew Glass for their many useful conversations and comments about this work. Thanks also go to Zachiri McKenzie and Philipp Kleppmann for discussions on the Axiom of Choice.

\section{Coset intersection graphs}\label{coset section}

A useful tool for studying the way left and right cosets interact, and 
obtaining transversals, is the coset intersection graph. In this section we re-state important results from our earlier work \cite{BCZ} on this concept. We denote the complete bipartite graph on $(m,n)$ vertices by $\comp_{m,n}$.

\begin{defn}
Let $H, K < G$. We define the \emph{coset intersection graph} $\Gamma^{G}_{H,K}$ to be a graph with vertex set consisting of all left cosets of $H$ $(\{l_{i}H\}_{i \in I})$ together with all right cosets of $K$ $(\{Kr_{j}\}_{j \in J})$, where $I$, $J$ are index sets. If a left coset of $H$ and right coset of $K$ correspond, they are still included twice. Edges (undirected) are included whenever any two of these cosets intersect, and an edge between $aH$ and $Kb$ (written $aH-Kb$) corresponds to the non-empty set $aH \cap Kb$.
\end{defn}

\begin{thm}\label{graph2}
Let $H,K <G$. Then the graph $\Gamma^{G}_{H,K}$ is a disjoint union of complete bipartite graphs. Moreover, suppose that $[G:H]=n$,  $[G:K]=m$. Then each connected component of $\Gamma^{G}_{H,K}$ is of the form \emph{$\comp_{s_{i}, t_{i}}$} with $s_{i}/t_{i} = n/m$.
\end{thm}

\begin{cor}\label{useful2}
Let $H,K<G$. Suppose that $[G:H]=n$ and $[G:K]=m$, where $m \geq n$. Then there exists a set $T \subseteq G$ which is a left transversal for $H$ in $G$, and which can be extended to a right transversal for $K$ in $G$. If $H=K$ in $G$, then $T$ becomes a left-right transversal for $H$.
\end{cor}

Under the hypothesis of Theorem \ref{graph2}, we see that sets of $s_{i}$ left cosets of $H$ completely intersect sets of $t_{i}$ right cosets of $K$, with $s_{i}/t_{i}$ constant over $i$. With this in mind, another way of visualising $\Gamma^{G}_{H,K}$ is by the following simultaneous double-partitioning $G$: draw left cosets of $H$ as columns, and right cosets of $K$ as rows, partitioning $G$ into irregular `chessboards' denoted $C_{i}$, each with edge ratio $n:m$. Each chessboard $C_{i}$ corresponds to the connected component $\comp_{s_{i}, t_{i}}$ of $\Gamma^{G}_{H,K}$, and individual tiles in $C_{i}$ correspond to the non-empty intersection of a left coset of $H$ and a right coset of $K$ (i.e., edges in $\comp_{s_{i}, t_{i}}$). Corollary \ref{useful2} would then follow by choosing one element from each tile on the leading diagonals of the $C_{i}$'s. An example of chessboards is given in \cite{BCZ}.

The chessboard pictorial representation of partitioning $G$ into left and right cosets is extremely useful in the analysis of transversals as generating sets carried out in the next section. Note that the union of all the elements of $G$ in a single chessboard gives a unique double coset $KgH$ in $G$, and that a single chessboard is simply a double-partitioning of a double coset $KgH$ into its respective left cosets of $H$ and right cosets of $K$.

\section{Transversals as generating sets}\label{shiftbox section}

We have developed a technique which we call \emph{shifting boxes} that, for the sake of brevity, we will describe here as a systematic way to apply Nielsen transformations to a generating set of a group $G$, such that the resulting generators lie inside (or outside) particular desired cosets of a subgroup $H<G$. We can't `shift' generators in/out of any coset we like, but we do have a substantial degree of control. For ease of notation, we will often refer to the coset $eH$ as the \emph{identity coset}. We begin with the following definitions.

\begin{defn}
Let $G$ be a group, and $S:=(g_{1}, \ldots, g_{n})$  
a \emph{generating $n$-tuple} of $G$ (where $n \in \mathbb{N}$), that is, an element of the direct product
$G^n$ such that $\{g_1,\ldots, g_n\}$ generates $G$. 
A \emph{standard Nielsen move} on $S$ is the replacement of some entry
$g_{i}$ of $S$ with one of $g_jg_i,g_j^{-1}g_i,g_ig_j$ or $g_ig_j^{-1}$, where
we must have $i\neq j$. A \emph{Nielsen move} is defined to be either
a standard Nielsen move or an \emph{extended Nielsen move}, where the
latter consists of either replacing an entry $g_i$ by its inverse, or
transposing two entries $g_i$ and $g_j$ for $i\neq j$. Note that on
applying any Nielsen move to $S$, the resulting $n$-tuple still generates.
\end{defn}

\begin{defn}
Let $G$ be a group. Two generating $n$-tuples
$S_{1}, S_{2}$ of $G$ are said to be \emph{Nielsen equivalent} 
if they differ by a finite number of Nielsen moves.
\end{defn}

\begin{defn}
Let $G$ be a group, $H$ a subgroup of $G$, and $S$ a generating
$n$-tuple of $G$. We say a left coset $gH$ is \emph{full} (with respect to $S$)
if some entry of $S$ lies in $gH$, otherwise we say $gH$ is \emph{empty} 
(with respect to $S$). To save on notation, we will usually suppress the 
term `with respect to $S$' when there is no ambiguity.
\end{defn}

In the introduction we mentioned that the Axiom of Choice was necessary to show that for every group $G$ and every subgroup $H<G$ there exists a left transversal for $H$ in $G$ (in fact, the Axiom of Choice is equivalent to this condition; see \cite[Theorem 2.1]{axiom}). However, choice is \emph{not} necessary for the groups we will be considering, as they are all finitely generated. To see this, take a finite generating set $\{x_{1}, \ldots, x_{n}\}$ for $G$ and form the canonical enumeration of words $w_{1}, w_{2}, \ldots$ on $X\cup X^{-1}$, ordered lexicographically. Then the set $T:=\{w_{n} \in G\ | \ (\forall i < n)(w_{i} \notin w_{n}H) \}$ is a left transversal for $H$ in $G$ (and we have not used choice here). Moreover, if we have a set $S \subset G$ for which no two elements of $S$ lie in the same left coset of $H$, then this set extends to a left transversal for $H$ by adjoining $T_{S}:=\{w_{n} \in G\ | \ (w_{n} \notin S) \wedge(\forall i < n)(w_{i} \notin w_{n}H) \}$ (again, without the need for choice). 

We now give several techniques, which we rely on heavily for our 
main results. Note that in this section we prove our results under very general conditions, and all techniques are (for now) existential. 

\begin{defn}
Let $H<G$ be groups. A $n$-tuple $S'$ with entries in $G$ is said to be \emph{left-cleaned} if all of its entries lie in distinct left cosets of $H$ (apart from $eH$ which may contain many entries of $S'$). An analogous definition applies for a \emph{right-cleaned} $n$-tuple, dealing with right cosets instead.
\end{defn}

\begin{lem}\label{clean}
Let $G$ be a group, $H$ a subgroup of $G$, and 
$S$ a generating $n$-tuple of $G$. Then there is a left-cleaned generating $n$-tuple $S'$ of $G$, Nielsen equivalent to $S$. An identical result holds for right cosets of $H$.
\end{lem}

\begin{proof}
We call the following process \emph{left-cleaning} an $n$-tuple.
Let $S=(g_{1}, \ldots, g_{n})$. We can assume that there is
$g_{i}, g_{j}$ with $i\neq j$
both lying in the same non-identity
left coset of $H$, so that $g_{i}H=g_{j}H\neq eH$. 
Then $g_{j}^{-1}g_{i} \in H$ so we can apply the standard Nielsen move on $S$
which replaces $g_{i}$ with $g_{j}^{-1}g_{i}$ to obtain $S_{1}$. 
Then $S_{1}$ has fewer entries lying in this left coset of $H$, and the
same number in all other non-identity left cosets. 
Iterating this procedure and then moving to other non-identity left cosets, 
we eventually reach a left-cleaned $n$-tuple $S'$.
\end{proof}

An analogous definition, result, and proof, applies for right cosets and 
right cleaning.

\begin{lem}\label{extract1}
Let $G$ be a group, $H$ a subgroup of $G$, and $S$ a generating 
$n$-tuple of $G$. If there exists at least one empty left coset of $H$, then 
there are entries $s_{j}, s_{k}$ of $S$ (possibly the same entry) and  $\epsilon \in \{\pm 1\}$ such that
$s_{j}^{\epsilon}s_{k}H$ is an empty left coset. That is, there is some full left coset of $H$ which is taken to some empty left coset of $H$ by left multiplication under some entry of $S$ or its
inverse.
\end{lem}

\begin{proof}
Recall that $G$ acts transitively on the set of left cosets by left 
multiplication. 
Assume that no entry of $S$ or its inverse
sends a full left coset to an empty left coset. 
Then, as the entries of $S$ generate $G$, the  
collection of full left cosets is invariant under this action.
Seeing as there exists at least one empty left coset, 
this contradicts the transitive action of $G$.
\end{proof}

\begin{lem}\label{extract2}
Let $G$ be a group, $H$ a subgroup of $G$, and $S$ a generating $n$-tuple of
$G$. Suppose that at least one entry of $S$ lies in $H$, and moreover that 
there exists an empty left coset of $H$ with respect to $S$. Then there is a 
finite sequence of Nielsen moves on some entry $s$ of $S$ which is
contained in $H$ such that $s$ is  
taken into an empty left coset of $H$.
\end{lem}

\begin{proof}
By Lemma \ref{extract1} there are (possibly identical) entries $s_{1}, s_{2}$ of $S$, and $\epsilon \in \{\pm 1\}$, with $s_{j}^{\epsilon}s_{k}H$ an empty left coset of $H$ with respect to $S$. 
We consider all possible cases:

$1$. The case $s_{j}, s_{k} \in H$ never occurs, as then $s_{j}^{\epsilon}s_{k}H = H$ which is a full 
left coset by hypothesis.

$2$. In the case $s_{j} \notin H, s_{k} \in H$, the subcase $s_{j}^{+1}s_{k}H$ can't occur, as then $s_{j}^{+1}s_{k}H= s_{k}^{+1}H$ which is clearly full. In the subcase $s_{j}^{-1}s_{k}H$, we 
replace $s_{k}$ with $s_{j}^{-1}s_{k}$ lying in the left coset $s_{j}^{-1}s_{k}H$ which is empty.

$3$. In the case $s_{j} \in H, s_{k} \notin H$, we replace $s_{j}$ with $s_{j}^{\epsilon}s_{k}$, as the left coset $s_{j}^{\epsilon}s_{k}H$ is empty.

$4$. In the case $s_{j}, s_{k} \notin H$, take some $s_{i} \in H$ and replace $s_{i}$ with
$s_{j}^{\epsilon}s_{k}s_{i}$, which lies in the empty left coset 
$s_{j}^{\epsilon}s_{k}s_{i}H=s_{j}^{\epsilon}s_{k}H$.
As $s_i$ is a different entry
from $s_j$ and $s_k$, this is a composition of two standard Nielsen moves
on the entry $s_i$ (even if $s_{j}=s_{k}$).
\\We call this replacement process a \emph{left-extraction} of an entry of $S$ 
from $H$.
\end{proof}

Using these techniques, we state the condition 
below for a finite index subgroup of a group to possess a left transversal which 
generates the whole group. 
For simplicity, when $S$ is an $n$-tuple, we write $\tilde{S}$ for the set of entries of $S$.

\begin{thm}\label{main}
Let $G$ be a finitely generated group, and $H$ a subgroup of finite index in $G$. Then the following are equivalent:
\\ $1$. $[G:H] \geq d(G)$.
\\ $2$. There exists a left transversal $T$ for $H$ in $G$ which contains a generating set $X$ for $G$ of size $|X|= d(G)$.
\end{thm}

As mentioned in the introduction, the above result also carries over to right transversals.

\begin{proof}
That $2\Rightarrow 1$ is immediate. We show $1\Rightarrow 2$:
\\Let $n=d(G)$, and let $S$ be a generating $n$-tuple for $G$. Use Lemma \ref{clean} to produce an $n$-tuple
$S''$ Nielsen-equivalent to $S$ which is left-cleaned. Now repeatedly apply 
Lemma \ref{extract2} to begin left-extracting elements from inside $H$ 
(thus Nielsen-transforming $S''$). As $n \leq [G:H]$, we can keep 
left-extracting until we reach $S'$ which is Nielsen-equivalent to our 
original $S$, and for which no two entries of $S'$ lie in the same left coset 
of $H$. Now simply choose one element from each left coset of $H$ which is empty with 
respect to $S'$, and add these to $\tilde{S}'$ to form the set $T$. Then $T$ is clearly a left transversal for $H$, and contains $\tilde{S}'$.
\end{proof}

A slight variant of the above proof also shows the following result: when $[G:H] \leq d(G)$ then there is a generating set for $G$ of size $d(G)$ which contains a set of left coset representatives for $H$.

For ease of writing, we will often refer to the overall process of cleaning 
and/or extracting elements (either left, or right) as \emph{shifting boxes}, 
and will usually just write \emph{this follows by shifting boxes} to mean 
that it follows by the process of cleaning and/or extracting elements. Our remarks in this section give sufficient conditions for cleaning and/or extracting to be algorithmic.

The most natural question to ask now is `When does a finite index subgroup 
have a left-right transversal which generates the whole group?' This requires a deeper understanding of how cosets intersect, as discussed in Section \ref{coset section}. We urge the reader to consider the discussion of `chessboards' given after Corollary \ref{useful2}, and to consult \cite{BCZ} for an example. These are vital in proving what follows.

Let $G$ be a group, $H$ a subgroup of $G$, 
and $S$ a generating $n$-tuple of $G$. Then by Lemma $\ref{clean}$
we can first perform a
left-cleaning of $S$ to form $S'$,
followed by a right-cleaning of $S'$ (which will stay left-cleaned)
to obtain a further generating $n$-tuple $S''$ such that:
\\$1$. $S$ and $S''$ are Nielsen-equivalent.
\\$2$. $S''$ is left-cleaned.
\\$3$. $S''$ is right-cleaned.
\\We say that $S''$ is \emph{left-right-cleaned}.

\begin{lem}\label{diagonals}
Let $G$ be a group, $H$ a subgroup of finite index in $G$, 
and $S$ a generating $n$-tuple of $G$. Then $S$ is left-right-cleaned if and 
only if one can draw chessboards for $H$ in $G$ with distinct entries of $S$ 
lying in distinct diagonal tiles of chessboards, except for the chessboard corresponding to the double coset
$HeH=H$ which may contain several elements of $S$.
\end{lem}

\begin{proof}
This is immediate from the fact that columns in chessboards correspond to left cosets of $H$, and rows correspond to right cosets. Thus, a column (resp.~row) in the chessboards contains multiple entries of $S$ if and only if the corresponding left (resp.~right) coset of $H$ contains multiple entries of $S$.
\end{proof}

Note that one can obtain a left-right transversal for $H$ by taking one 
element from each diagonal tile of each chessboard (by Corollary \ref{useful2}). More strongly, 
by left-right-cleaning and choosing an element from each unused diagonal
we have:

\begin{lem}\label{extend LR}
Let $G$ be a group, $H$ a subgroup of finite index in $G$, 
and $S$ a generating $n$-tuple of $G$. If $S$ is left-right-cleaned, and $H$ 
contains at most one entry of $S$, then there is a set $T$ containing
all the entries of $S$ which is a left-right transversal for $H$ in $G$.
\end{lem}

\begin{proof}
Given that columns in chessboards correspond to left cosets of $H$, and rows correspond to right cosets, we have that no column or row in any chessboard contains more than one entry from $S$. Thus we can re-arrange the positioning of the columns and rows in each chessboard so that the entries of $S$ are all in tiles which lie on leading diagonals. Now simply choose one element from each lead-diagonal tile which does not contain an entry of $S$, and add these to the set $S$ to form the set $T$. Then $T$ contains precisely one element from each lead-diagonal tile of each chessboard, and no other elements. Thus $T$ contains precisely one element in each left coset of $H$, and precisely on element in each right coset of $H$. So $T$ is our desired left-right transversal which contains $\tilde{S}$.
\end{proof}

Combining our shifting boxes technique, along with the properties of the coset 
intersection graph from Theorem \ref{graph2}, we are able to show the following:

\begin{thm}\label{rank3 pre}
Let $G$ be a group, $S$ a generating $n$-tuple for $G$ with $n \leq 3$, and $H$ a subgroup of finite index in $G$ with $n \leq [G:H]$. Then there is a generating $n$-tuple $S'$ Nielsen-equivalent to $S$, and a left-right transversal $T$ for $H$ in $G$ with $\tilde{S}'\subseteq T$.
\end{thm}

\begin{proof}
The case when $n=1$ is trivial. The case when $n=2$ is done as 
follows:
\\Left-right-clean $S$ to form 
$S'=(a,b)$. Clearly we can't have $a, b \in H$, or else $H$ can't have index 
$\geq 2$. So at most one of $a,b$ lies inside $H$. But then by 
Lemma \ref{extend LR} we can extend $\tilde{S}'$ to a set $T$ which is a left-right 
transversal for $H$. Seeing as $T$ contains $\tilde{S}'$, then it generates $G$.

The case where $n=3$ is much more complicated, and we need to consider 
several sub-cases. So, left-right-clean $S$ to form $S'=(a,b,c)$. Clearly we cannot have 
$a, b,c \in H$, or else $H$ does not have index at least $3$. So at most two of 
$a,b$ lie inside $H$. Again, if only one of $a,b,c$ lies inside $H$ then we 
can apply Lemma \ref{extend LR} as before. So we are left to consider what 
happens when two of $a,b,c \in H$ (without loss of generality, re-label them as 
$h_{1}, h_{2} \in H$ and $g \notin H$). 

Case 1. $g^{2} \notin HgH \cup H$ (i.e., $g^{2}$ lies in a different 
chessboard to $g$ and $h_{1}, h_{2}$).
\\Make the Nielsen moves $h_{1} \mapsto g^{2}h_{1}$; this clearly lies in the same left coset (and hence same chessboard) as $g^{2}$ (see Figure \ref{fig1}).

\begin{figure}[ht]
\resizebox{8cm}{!}{\includegraphics{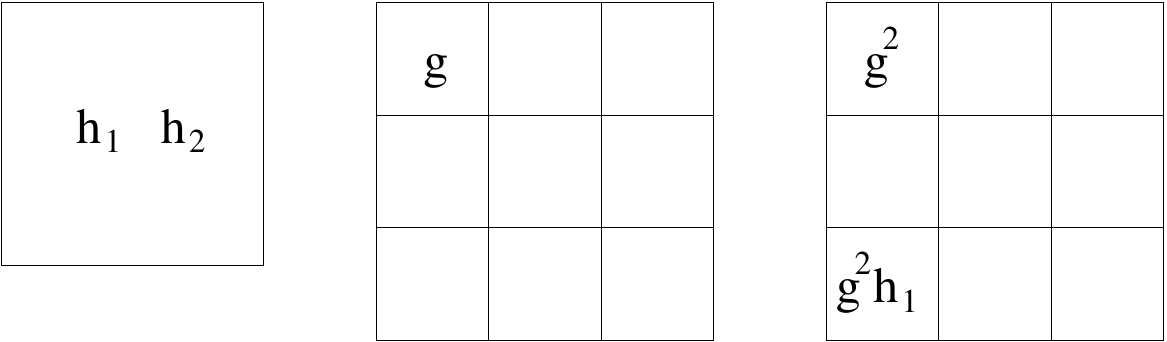}}
\caption{}
\label{fig1}
\end{figure}

So now each of $g, h_{2}, g^{2}h_{1}$ lie in different chessboards, thus the triple $S'':=(g, h_{2}, g^{2}h_{1})$ is left-right cleaned. As only $h_{2}$ lies inside $H$, we can use Lemma \ref{extend LR} to extend $\tilde{S}''$ to a set $T$ which is a left-right transversal for $H$. Seeing as $T$ contains $\tilde{S}''$, then it generates $G$.
\\If case 1 does not occur, then we proceed to case 2.

Case 2. $g^{2} \in HgH$ (i.e., $g^{2}$ lies in the same chessboard as $g$).
\\Clearly $g^{2}H \neq gH$ and $Hg^{2} \neq Hg$; otherwise we would have $g \in H$ which contradicts our initial hypothesis. So $g^{2}$ lies in a different left coset and different right coset to $g$ (i.e., in a different column and row to $g$ in $HgH$). Consider $h_{1}g^{2}$ and $h_{2}g^{2}$ (which both lie in the same right coset as $g^{2}$, and hence in a different right coset to $g$). If $h_{i}g^{2}H \neq gH$ for some $i \in \{1,2\}$, then make the Nielsen moves $h_{i} \mapsto h_{i}g^{2}$ which lies in a different left and different right coset to $g$ (but in the same chessboard) (see Figure \ref{fig2}).

\begin{figure}[ht]
\resizebox{6cm}{!}{\includegraphics{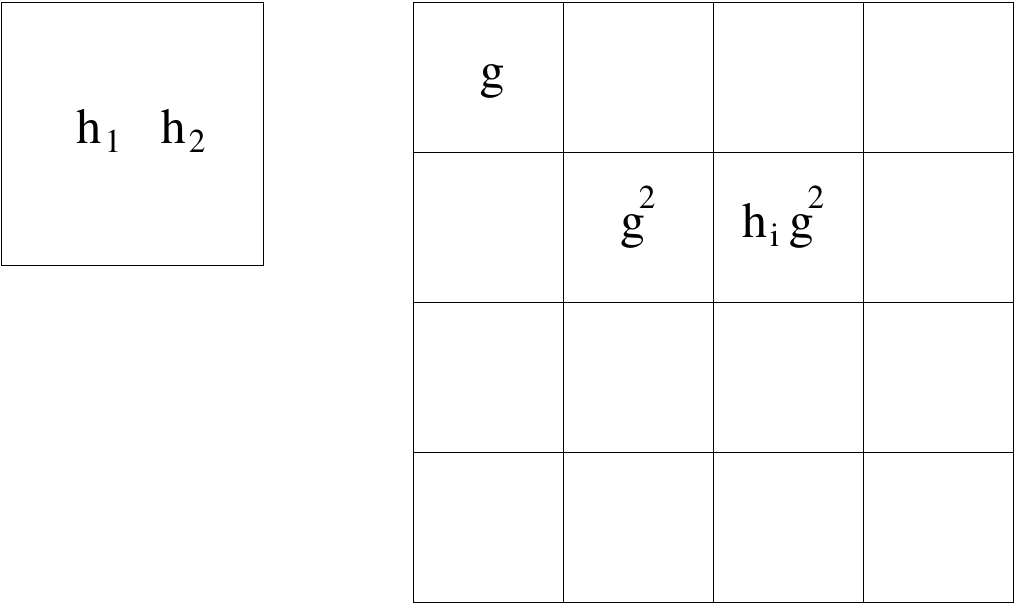}}
\caption{}
\label{fig2}
\end{figure}

If on the other hand $h_{1}g^{2}H=h_{2}g^{2}H=gH$, then $h_{2}^{-1}h_{1}g^{2}H=g^{2}H$ and so we make the Nielsen moves $h_{1} \mapsto h_{2}^{-1}h_{1}g^{2}$ which lies in a different left and different right coset to $g$ (but in the same chessboard) (see Figure \ref{fig3}).

\begin{figure}[ht]
\resizebox{6cm}{!}{\includegraphics{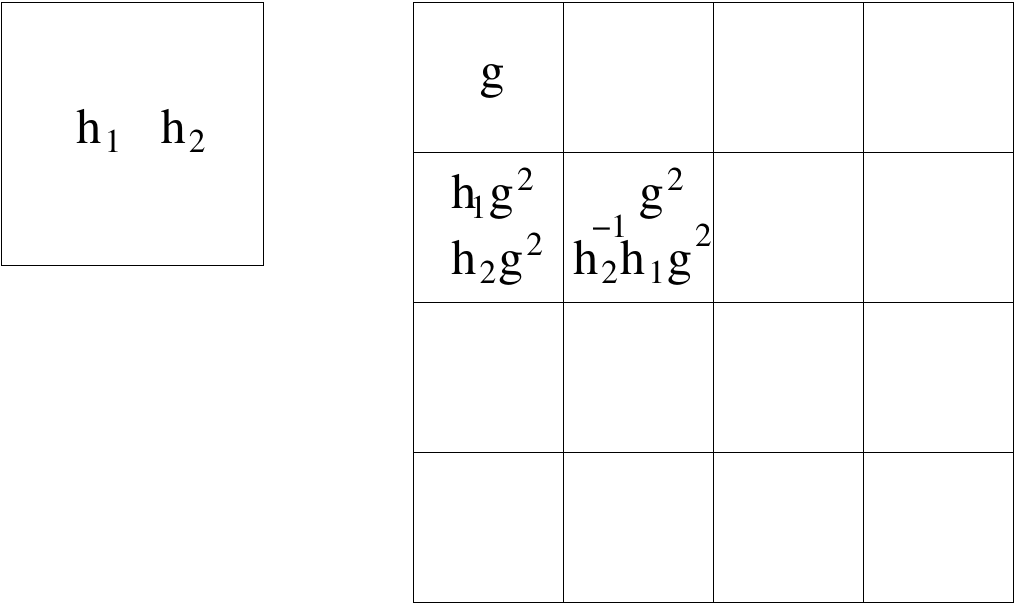}}
\caption{}
\label{fig3}
\end{figure}

Either way, we now have a triple $S''$ for which,  after permutation of some rows and columns, has entries which lie along diagonal tiles of the chessboards.  As only one such entry lies inside $H$, we can use Lemma \ref{extend LR} to extend $\tilde{S}''$ to a set $T$ which is a left-right transversal for $H$. Seeing as $T$ contains $\tilde{S}''$, then it generates $G$.
\\If neither case 1 nor case 2 occur, then we proceed to case 3.

Case 3. $g^{2} \in H$.
\\By the transitivity of the action of $G$ on left (and right) cosets of $H$, there must be some $h_{i}$ ($i \in \{1,2\}$) and some $\epsilon \in \{\pm 1\}$ with $h_{i}^{\epsilon}gH \neq gH$, and similarly some $h_{j}$ ($j \in \{1,2\}$) and some $\delta \in \{\pm 1\}$ with $Hgh_{j}^{\delta} \neq Hg$. If $i\neq j$, then we make the Nielsen move $h_{i} \mapsto h_{i}^{\epsilon}g$ followed by the Nielsen move $g \mapsto gh_{j}^{\delta}$ (see Figure \ref{fig4}). 

\begin{figure}[ht]
\resizebox{6cm}{!}{\includegraphics{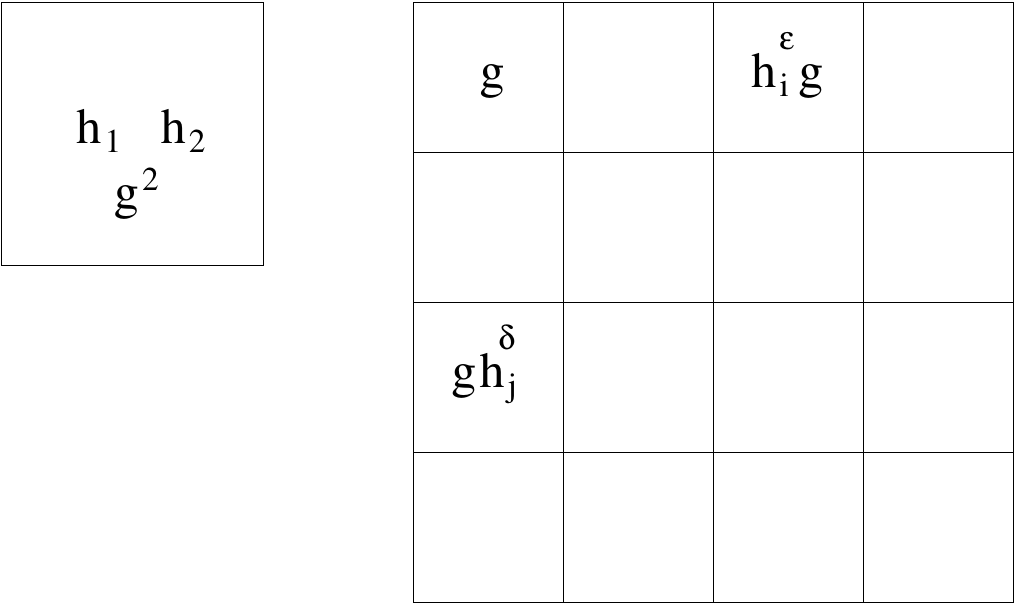}}
\caption{}
\label{fig4}
\end{figure}

If on the other hand $i=j$ (say $i=j=1$, without loss of generality), then consider the element $h_{2}gh_{1}^{\delta}$. If $h_{2}gh_{1}^{\delta}H \neq gH$, then $h_{2}gh_{1}^{\delta}$ lies in a different left coset and different right coset to $g$, and so we make the Nielsen moves $h_{2} \mapsto h_{2}gh_{1}^{\delta}$ (see Figure \ref{fig5}).

\begin{figure}[ht]
\resizebox{6cm}{!}{\includegraphics{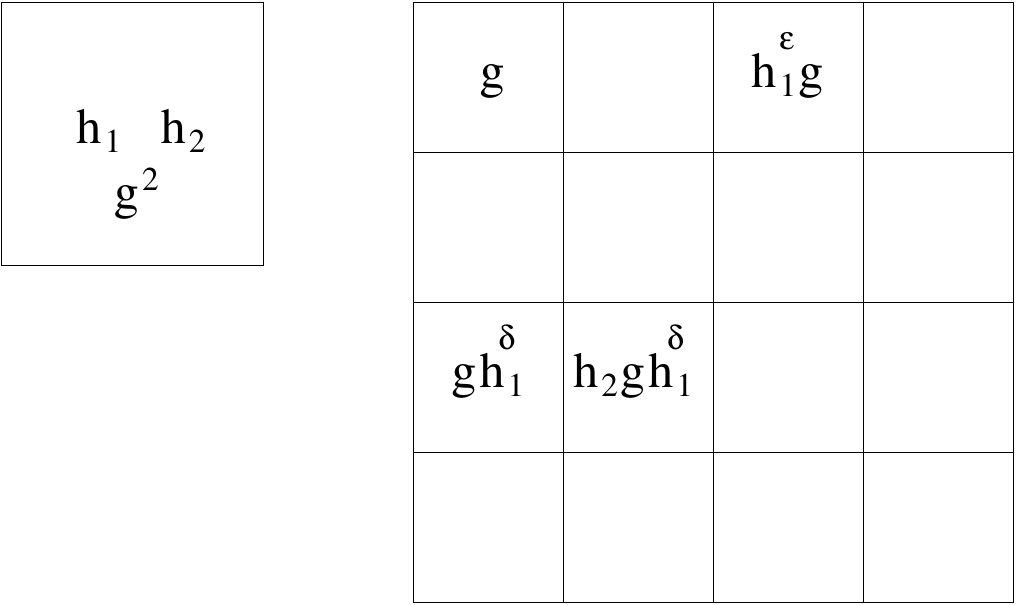}}
\caption{}
\label{fig5}
\end{figure}

If however $h_{2}gh_{1}^{\delta}H = gH$, then $h_{2}gh_{1}^{\delta}$ lies in a different left coset and different right coset to $h_{1}^{\epsilon}g$, and so we make the Nielsen moves $h_{2} \mapsto h_{2}gh_{1}^{\delta}$ followed by $g \mapsto h_{1}^{\epsilon}g$ (see Figure \ref{fig6}).

\begin{figure}[ht]
\resizebox{6cm}{!}{\includegraphics{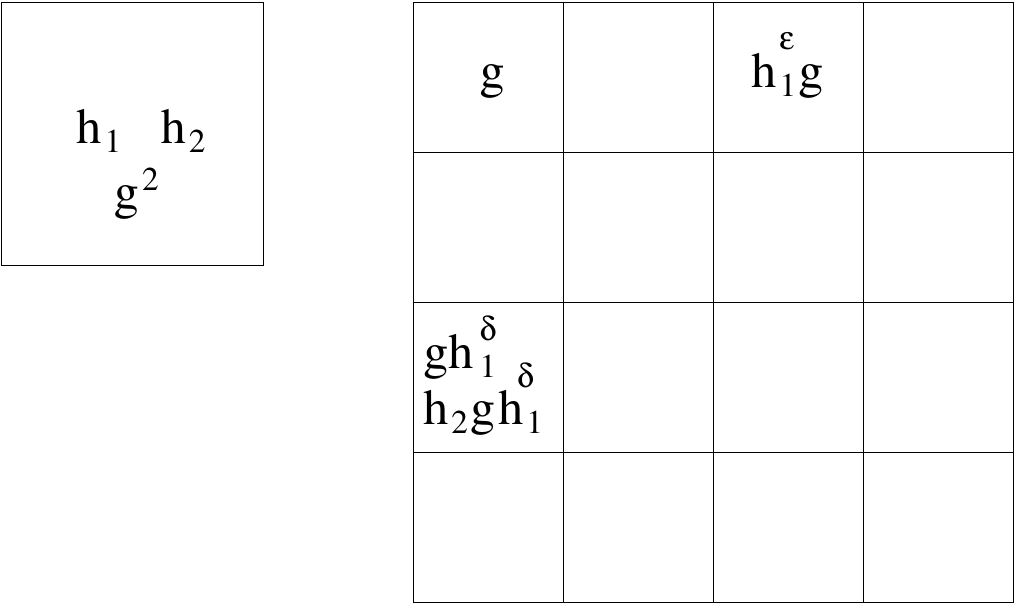}}
\caption{}
\label{fig6}
\end{figure}

In all the subcases considered here, we end up with a triple $S''$ for which,  after permutation of some rows and columns, has entries which lie along diagonal tiles of the chessboards.  As only one such entry lies inside $H$, we can use Lemma \ref{extend LR} to extend $\tilde{S}''$ to a set $T$ which is a left-right transversal for $H$. Seeing as $T$ contains $\tilde{S}''$, then it generates $G$.
\end{proof}

Given the existential nature of the results stated in this section, it is always possible (via naive searches) to algorithmically construct the relevant generating tuples and transversals mentioned in the results. However, by following our proofs more closely, one can algorithmically construct these in a manner much faster than mere naive searches.

\begin{thm}\label{rank3}
Let $G$ be a group with $d(G) \leq 3$, and $H$ a subgroup 
of finite index in $G$. Then the following are equivalent:
\\ $1$. There exists a left-right transversal $T$ for $H$ in $G$ with $\langle T \rangle=G$.
\\ $2$. $[G:H] \geq d(G)$.
\end{thm}

\begin{proof}
That $1\Rightarrow 2$ is immediate; $2\Rightarrow 1$ can be seen from Theorem \ref{rank3 pre}.
\end{proof}

This leads us to pose the following question:

\begin{ques}
 Does Theorem \ref{rank3}  hold if we change the hypothesis `$d(G)\leq 3$' to `$d(G)$ finite'?
\end{ques}

In the  proof of Theroem \ref{rank3 pre}, each case $n=1,2,3$ is shown by analysing a 
(increasing) finite number of possible scenarios, via repeated application of 
shifting boxes. We have not yet extended this to the case $n\geq 4$, as 
the number of scenarios to consider becomes very large and complex. We believe 
that the most elegant way to do this would be to establish some technique that,
given a left-right-cleaned generating $n$-tuple $S$ of $G$ with more than one 
entry of $S$ lying in $H$, left-right-extracts an entry of $S$ lying in $H$ 
to an empty square on the diagonal of one of the chessboards. We have not been 
able to do this, but believe that it may indeed be possible.

There is an equivalence between our \emph{shifting boxes} technique, and the Schreier graph of the action of $G$ on $G/H$ with respect to a generating set for $G$. Though the Schreier graph could be a valid approach to take, it seems  more cumbersome to work with for the problems we have considered, and proofs using Schreier graphs appear to be longer and less intuitive than what we have done here.

We note that, in any extension of Theorem \ref{rank3} to groups needing
more than three generators,
we need only consider free groups, as the following shows:

\begin{prop}\label{reduce to free}
Theorem $\ref{rank3}$ holds for all finite rank groups $($rather than just groups of rank at most $3)$ if and only if it holds for all finite rank free groups.
\end{prop}

\begin{proof}
Suppose Theorem \ref{rank3} holds for all finite rank free groups. Let $G$ be 
a group with $d(G)=n$, and $H<G$ a subgroup of index $[G:H]=k \geq n$. Then 
there is a surjection $f: F_{n} \twoheadrightarrow G$, and it is a standard
fact that the preimage $f^{-1}(H)$ also has index $k$ in $F_{n}$.
By hypothesis, there is a left-right transversal $T$ of $f^{-1}(H)$ which 
generates $F_{n}$; it follows that $f(T)$ is a left-right transversal of $H$ 
which generates $G$.
\end{proof}

Thus we are able to restrict our investigation to the case of finite rank free 
groups. By doing this, we are able to make use of the rich theory of Stallings 
graphs, which are precisely Schreier graphs for subgroups of free groups \cite{KapMia}. This simplification to free groups also extends to all earlier results in this section before Theorem \ref{rank3} (they are true for all 
groups if and only if they are true for free groups). With the help of Enric 
Ventura and Jordi Delgado we have replicated proofs of our 1-sided transversal results (Lemmata \ref{clean}--\ref{extract2} and Theorem \ref{main}) using the framework of Stallings graphs. However, we stress that these are more complicated, 
and our attempts to generalise (or even re-prove) Theorems \ref{rank3 pre} and \ref{rank3} using Stallings graphs have so far been unsuccessful.

\section{An application of shifting boxes: finding primitive elements}

Recall that a primitive element of a finite rank free group $F_{n}$ is 
one which lies in \emph{some} generating set of size precisely $n$, which
is equivalent to being an element of a free basis for $F_n$. 
If $G$ is a group of finite rank $n$, then we say a \emph{primitive element} in $ G$ is an element lying in \emph{some} generating set of size $n$ for $G$. This reduces to the standard definition of primitive elements in finitely generated free groups.

An obvious question to ask is which subgroups of $F_n$ (or more generally, rank $n$ groups) contain a primitive element
(we can ask this for both finite and infinite index subgroups).
We first consider the case of normal subgroups.

The following is immediate by considering the image under the natural
homomorphism of a generating set of minimal size containing
the relevant primitive element:

\begin{lem}\label{rank}
Let $G$ be a group of finite rank $n$, and $N$ 
a normal subgroup of $G$. If $N$ contains some primitive element of $G$, then 
$d(G/N) <n$.
\end{lem}

The converse statement is not true, even in the special case that $G=F_{n}$, as was shown in \cite{Nos} and in
\cite{Ev} when $N$ has infinite index. It is currently open
if $N$ has finite index and here we briefly mention the
connection with product replacement graphs. Much more can
be found in the survey article
\cite{Pak} of Pak which contains a range of references.

Given a finitely generated group $G$ and an integer $n\geq d(G)$, the
product replacement graph $\Gamma_n(G)$ has vertices the generating
$n$-tuples of $G$ with edges between two vertices if one is the image of
another under a standard Nielsen move. 
A big area of study in this topic is the connectivity of
$\Gamma_n(G)$. 
Now if $n\geq d(G)+\overline{d}(G)$, where $\overline{d}$ is the
maximum size of a \emph{minimal} generating set (one in which no proper 
subset is a generating set), then it is known that
$\Gamma_n(G)$ is connected. However $\overline{d}$
can be a very hard quantity to evaluate in practice: indeed the result
of Whiston mentioned earlier which was used by Cameron is that
$\overline{d}(S_n)=n-1$ and the proof needs CFSG. 
It can happen that $\Gamma_n(G)$ is
disconnected when $n=d(G)$ (for instance finite abelian groups)
but no example is known of a finite group $G$ and an integer $n>d(G)$
where $\Gamma_n(G)$ is disconnected. 
If $G$ is finite soluble and
$n>d(G)$ then Dunwoody showed in \cite{Dunw} that $\Gamma_n(G)$ is
connected. The relation with primitive elements, as noted by
Dunwoody in that paper, is that if $N$ is a normal subgroup of $F_n$
containing no primitive element but $G=F_n/N$ has $d(G)<n$ then
$\Gamma_n(G)$ is disconnected. Hence
there are examples of infinite finitely generated groups $G$ with
$\Gamma_n(G)$ disconnected by the papers cited above, but a normal
subgroup $N$ of finite index 
containing no primitive element and with $d(F_n/N)<n$ would give rise to
a finite group $G$ and integer $n>d(G)$ with $\Gamma_n(G)$ disconnected; the
existence of which is currently unknown.

Our shifting boxes technique enables us to explore the location of primitive elements relative to cosets of a finite index subgroups, in the following ways. We write $[n]$ for the set of integers $\{1, \ldots, n\}$, and $X \Delta Y$ for symmetric difference.

\begin{lem}\label{primitive1 pre}
Let $G$ be any group with $d(G)=n$,
and $H$ a subgroup of finite index in $G$ 
with $[G:H] < 3 \cdot 2^{n-1}$. If $H$ contains no primitive elements of $G$, then $H$ contains the square of every primitive element of $G$.
\end{lem}

\begin{proof} 
Let $S=(g_{1}, \ldots, g_{n})$ be any left-cleaned generating $n$-tuple for $G$. 
For any $\emptyset \neq M =\{i_{1}, \ldots,  i_{k}\} \subseteq [n]$, ordered so $i_{1} < \ldots < i_{k}$, define the unique word  $w_{M}:= g_{i_{k}}g_{i_{k-1}} \cdots g_{i_{1}}$ Define the disjoint sets of words $A:= \{w_{M}\ |\ \emptyset \neq M \subseteq \{2, \ldots, n\} \textnormal{ or } M=\{1\}\}$, $B:=\{w_{M}g_{1}\ |\ \emptyset \neq M \subseteq [n]\}$. 
Now set $T:=A \sqcup B$, and thus $|T|=  3 \cdot 2^{n-1}-1$. By construction, the only element in $T$ which might not be primitive is $g_{1}^{2}$, and the rest are primitive by Nielsen transformations: for any $w_{M}g_{1} \in B$ with $M \neq \{1\}$ we take some $1 \neq i \in M$ and perform the Nielsen transformation $g_{i} \mapsto w_{M}g_{1}$ ($g_{i}$ appears precisely once in $w_{M}g_{1}$), a similar argument works for any $w_{M} \in A$.

We claim that, for any pair of distinct words $x,y \in T$, if $xH=yH$ then either $H$ contains a primitive element or $g_{1}^{2}\in H$ (possibly both). We consider all cases:

1. $ x,y \in A$, so $x=w_{M}, y=w_{M'}$, with $M \neq M'$. Then $w_{M}^{-1}w_{M'}\in H$ is primitive, as there is some $i \in M \Delta M'$ so Nielsen transform $g_{i} \mapsto w_{M}^{-1}w_{M'}$.

2. Precisely one of $x,y$ lie in $A$ (say $x \in A$), so $x=w_{M}$ and $y=w_{M'}g_{1}$. We consider all subcases: 
2A) $1 \notin M, M'$. In this case, $w_{M}^{-1}w_{M'}g_{1}\in H$ is primitive (Nielsen transform $g_{1} \mapsto w_{M}^{-1}w_{M'}g_{1}$). 
2B) $1 \in M'$, $1 \notin M$, and $M' = M \cup \{1\}$. In this case, $w_{M}^{-1}w_{M'}g_{1}=g_{1}^{2}\in H$. 2C) $1 \in M'$, $1 \notin M$, and there is some $1<j \in M\Delta M'$. In this case, $w_{M}^{-1}w_{M'}g_{1}\in H$ is a primitive element, as we can Nielsen transform $g_{j} \mapsto w_{M}^{-1}w_{M'}g_{1}$. 
2D) $M=\{1\}$.  If $M'=\{1\}$ then $w_{M}^{-1}w_{M'}g_{1}=g_{1}\in H$ is a primitive element. Otherwise, there is some $1 \neq j \in M$, in which case $w_{M}^{-1}w_{M'}g_{1}\in H$  is a primitive element as we can Nielsen transform $g_{j} \mapsto w_{M}^{-1}w_{M'}g_{1}$.

3. $x,y \in B$. In this case, $x=w_{M}g_{1}$, $y=w_{M'}g_{1}$ ($M\neq M'$). If there is some $1\neq i \in M \Delta M'$, then the element $g_{1}^{-1}w_{M}^{-1}w_{M'}g_{1}\in H$ is primitive. Otherwise, $M'=M\cup \{1\}$ (or vice-versa), in which case $g_{1}^{-1}w_{M}^{-1}w_{M'}g_{1}=g_{1}^{\pm 1}\in H$ is primitive.

Suppose that $H$ contains no primitive element. Since $|T| = 3 \cdot 2^{n-1}-1 \geq [G:H]$ then either two elements of $T$ lie in the same coset (so by the claim above, $g_{1}^{2}\in H$ as $H$ contains no primitive element), or one element of $T$ lies in $H$ (which must be $g_{1}^{2}$, as all other elements of $T$ are primitive). So $g_{1}^{2} \in H$.

Now, take \emph{any} primitive element $x\in G$, which is part of some generating set $\{x, y_{2}, \ldots, y_{n}\}$ for $G$ (which must be left-cleaned, otherwise $H$ would contain a primitive element). Using the exact same argument above, with $g_{1}:=x$, $g_{i}:=y_{i}$ for all $2 \leq i \leq n$, we see that if $H$ contains no primitive element then $x^{2} \in H$. So $H$ contains the square of \emph{every} primitive element.
\end{proof}

We can strengthen Lemma \ref{primitive1 pre} as follows:

\begin{lem}\label{primitive1}
Let $G$ be any group with $d(G)=n$,
and $H$ a subgroup of finite index in $G$ 
with $[G:H] < 3 \cdot 2^{n-1}$. If $H$ contains no primitive elements of $G$, then $H$ is normal in $G$ and $G/H \cong C_{2}^{m}$ for some $m \leq n$.
\end{lem}

\begin{proof}
Suppose $H$ contains no primitive element. Then, by Lemma \ref{primitive1 pre}, $H$ contains the square of every primitive element. Set $T:=\{g^{2}\ |\ g \textnormal{ is primitive in }G \}$. Then $T$ is a normal subset of $G$, since the conjugate of a primitive element is again a primitive element (conjugation is an automorphism). Thus $\langle  T  \rangle \vartriangleleft G$; the (normal) subgroup generated by all the squares of primitive elements of $G$. So by hypothesis, $T< H$. Now take any generating set $\{t_{1}, \ldots, t_{n}\}$ for $G$, then $t_{i}^{-1}t_{j}$ is a primitive element, for any pair $i, j$ with $i \neq j$. Thus $t_{i}^{2}, t_{j}^2, (t_{i}^{-1}t_{j})^{2}$ lie in $\langle  T  \rangle$, and hence so will $t_{i}^{2} (t_{i}^{-1}t_{j})^{2}  t_{j}^{-2} = [t_{i}, t_{j}] $. So $G/\langle  T  \rangle$ is abelian as $\langle  T  \rangle$ is normal and contains the commutator of every pair in the generating set $\{t_{1}, \ldots, t_{n}\}$ for $G$. So $\langle  T  \rangle $, and hence $H$, contain the commutator subgroup $[G,G]$. Thus $H$ is normal in $G$ and $G/H$ is generated by the images of $\{t_{1}, \ldots, t_{n}\}$, all of which have order 2 in this quotient. So $G/H  \cong C_{2}^{m}$ for some  $m \leq n$.
\end{proof}

We now give the following complete characterisation of finite index subgroups of a group of rank $n$ 
which contain primitive elements, up to index $3 \cdot 2^{n-1}-1$.

\begin{thm}\label{classify} 
Let $G$ be any group with $d(G)=n$ and
let $H$ be a subgroup of finite index in $G$ with $[G:H] < 3 \cdot 2^{n-1}$. 
Then $H$ contains no primitive elements of $G$ if and only if $H$
is normal in $G$ and the quotient $G/H$ is isomorphic to $C_{2}^{n}$,
whereupon every coset distinct from $H$ contains a primitive element of $G$.
\end{thm}

\begin{proof}
First if $H$ is normal and
contains an element $g$ of
a generating $n$-tuple for $G$ then the image of this $n$-tuple
gives rise to a generating $(n-1)$-tuple of $G/H$, just as in
Lemma \ref{rank}, but $d(C_2^n)=n$. 

Now suppose that $H$ does not contain a primitive element of $G$
and let $q:G \twoheadrightarrow G/H$ be the quotient homomorphism, where
we know $H$ is normal in $G$ and $G/H\cong C_{2}^{m}$ for some $m \leq n$ 
by Lemma \ref{primitive1}. 
Given a generating $n$-tuple $(g_1,\ldots ,g_n)$ for $G$, let
$F_n$ be the free group on $x_1,\ldots ,x_n$ and set $\theta : F_{n} \twoheadrightarrow G$ to be the
homomorphism  extending the map $x_i\mapsto g_i$. 
Note that if we have $k\leq n$ and integers
$1\leq i_1<i_2<\dots <i_k\leq n$ then $x_{i_1}x_{i_2}\cdots x_{i_k}$ is
primitive in $F_n$ and $\theta(x_{i_1}x_{i_2}\cdots x_{i_k})
=g_{i_1}g_{i_2}\cdots g_{i_k}$ is primitive in $G$.

Assume that $m<n$ and consider the map $q\circ \theta : F_n \twoheadrightarrow C_2^m$, which factors
through $C_2^n$ via the abelianisation map $\ab: F_{n} \twoheadrightarrow C_{2}^{n}$ and the map $\psi: C_{2}^{n} \twoheadrightarrow C_{2}^{m}$. That is, the following diagram commutes, and all maps are surjections:

\begin{displaymath}
\xymatrix{
F_{n} \ar[r]^{\ab}  \ar[d]_{\theta} & C_{2}^{n} \ar[d]^{\psi} \\
 G \ar[r]_{q} & C_{2}^{m}
}
\end{displaymath}

As $\psi$ is now a linear map from an $n$ dimensional
vector space over $\mathbb F_2$ to an $m$ dimensional space, we have a non-trivial
element $(v_1,\ldots ,v_n)$ of $C_2^n$ in the kernel of $\psi$. Now we can
assume that each $v_i$ takes the value 0 or 1, so we form the primitive
element $x=x_1^{v_1}x_2^{v_2}\cdots x_n^{v_n}$ of $F_n$
which maps to the identity under $\psi \circ \ab$, thus
$g_1^{v_1}g_2^{v_2}\cdots g_n^{v_n}=\theta(x)$ is a primitive element of $G$
which maps to the identity under $q$ and so is in $H$; a contradiction. 

Similarly if $n=m$ then, for any $(w_1,\ldots ,w_n)$ in 
$\mathbb F_2^n -\{0\}$, the coset of $H$ in $G$ corresponding to this
point contains the primitive element 
$g_1^{w_1}g_2^{w_2}\cdots g_n^{w_n}$ of $G$.
\end{proof}

Thus if $G$ is a group with $d(G)=n$ we have two possibilities: either
$C_2^n$ is not a quotient of $G$ in which case
all subgroups of $G$ having index less than $3 \cdot 2^{n-1}$ contain primitive
elements, or $G$ surjects to $C_2^n$ in which case there is a single
subgroup of index less than $3 \cdot 2^{n-1}$ which fails to contain a primitive
element. The uniqueness in the second case comes about because
a homomorphism from a rank $n$ group $G$ to an abelian group of
exponent 2 must factor through $G/[G,G]G^2$. As $G/([G,G]G^2)\cong C_2^m$
for $m\leq n$, we see that when $n=m$ any exceptional subgroup must be
equal to $[G,G]G^2$. 

Note that the inequality $[G:H] < 3 \cdot 2^{n-1}$ in Theorem \ref{classify} is somewhat necessary: here is an example of what occurs when the inequality doesn't hold.

\begin{exmp}\label{2x2n eg}
Take the (free) subgroup $H:=[F_{n},F_{n}]F_{n}^2<F_{n}$ of index $2^{n}$ with no primitive elements of $F_{n}$. Then $H$ itself has several normal subgroups of index 2, none of which contain primitive elements of $F_{n}$. In $F_{n}$, these subgroups have index $2\cdot2^{n}>3 \cdot 2^{n-1}$, so this is not a counterexample to Theorem \ref{classify}.
\end{exmp}

We remark that the number of subgroups of $F_n$ with
index less than $3 \cdot 2^{n-1}$ is vast: for instance
by \cite[Corollary 2.1.2]{LubSeg} the number
of subgroups of $F_n$
with index equal to $2^n$ is bounded below by $((2^n)!)^{n-1}$,   
yet only one of these subgroups fails to contain a primitive element by Theorem
\ref{classify}.

It would be interesting to find a closed form expression for $M(n)$, which we define to
be the smallest number $i$ such that $F_{n}$ has a subgroup other than $[F_{n},F_{n}]F_{n}^2$ of index $i$ which contains no primitive elements. By Theorem \ref{classify} and the example immediately proceeding it we have $3 \cdot 2^{n-1}\leq M(n) \leq 2\cdot2^{n}$.  In particular, consider any quotient map $f:F_{2} \twoheadrightarrow S_{3}$; the kernel $N$ of this map has index $6=3 \cdot 2^{2-1}$, and moreover $N$ contains no primitive elements of $F_{2}$ by Lemma \ref{rank}. So $M(2)=6$. Moreover, it is straightforward to see that $M(1)=3$. We do not know $M(n)$ for any other values of $n$.

Our analysis of the possible location of primitive elements, relative to 
finite index subgroups of $F_{n}$, was motivated by the following result of Parzanchevski 
and Puder in \cite{Puder}:

\begin{thm}[{\cite[Corollary 1.3]{Puder}}]
The set $P$ of primitive elements in $F_{n}$ is closed in the profinite 
topology.
\end{thm}

\begin{cor} \label{incoset} 
Given $F_{n}$, and $w\in F_{n}$ a non-primitive element, there is a finite index subgroup $H<F_{n}$ such that the coset $wH$ does not contain any primitive elements (but of course contains $w$). Taking $w=e$ gives a finite index
subgroup with no primitive elements.
\end{cor}

Given that the above result is existential, we decided to apply our techniques to look for explicit examples of subgroups with no primitive elements; Lemma \ref{primitive1} and Theorem \ref{classify} came about from this analysis, somewhat serendipitously. We have shown that $H=[F_{n},F_{n}]F_{n}^2$ is the (unique) finite index subgroup of $F_{n}$ of smallest index to contain no primitive elements, and that the next such example must have index $i \geq 3 \cdot 2^{n-1}$ (with equality in the case $n=2$).

We finish by remarking that a recent result of Clifford and Goldstein
in \cite{CandG} proves
there is an algorithm to determine whether or not a 
finitely generated subgroup of $F_n$ contains 
a primitive element, although they say that they do not expect it to
be implemented in practice. One of our overall aims is to give a characterisation of such subgroups that leads to computationally-efficient recognition.

\vspace{5pt}

\noindent \scriptsize{\textsc{Selwyn College, Cambridge
\\Grange Road, Cambridge, CB3 9DQ, UK
\\J.O.Button@dpmms.cam.ac.uk
\\
\\Mathematics Department, University of Neuch\^{a}tel
\\Rue Emile-Argand 11, Neuch\^{a}tel, CH-2000, SWITZERLAND
\\maurice.chiodo@unine.ch
\\
\\31 Mariner's Way, Cambridge, CB4 1BN,  UK
\\marianozeron@gmail.com}

\end{document}